\documentclass[11pt]{article}
\usepackage{amsmath}
\usepackage{amsthm}
\usepackage{upref}
\usepackage{amssymb}

\usepackage{verbatim}
\numberwithin{equation}{section}

\renewcommand{\labelenumi}{\rm (\theenumi)}

\newcommand{\A}{{\mathcal A}}

\newcommand{\N}{{\mathbb N}}
\newcommand{\R}{{\mathbb R}}

\newcommand{\HH}{{\mathcal H}}
\newcommand{\HHH}{{\mathbb H}}
\newcommand{\LL}{{\mathcal L}}

\newcommand{\rr}{{\mathbf r}}

\newcommand{\rom}{\romannumeral}
\def\Vect#1{\mbox{\boldmath$d$} }

\DeclareMathOperator{\Var}{Var}
\DeclareMathOperator{\loc}{loc}
\newcommand{\pd}[2]{\displaystyle{\frac{\partial#1}{\partial#2}}}
\newcommand{\od}[2]{\displaystyle{\frac{d#1}{d#2}}}

\newcommand{\norm}[1]{\Vert#1\Vert}

\newcommand{\norms}[1]{\Vert#1\Vert^2}

\newcommand{\Norm}[2]{{\left\|#1\right\|_#2}}

\newcommand{\Norms}[2]{{\left\|#1\right\|^2_#2}}

\newtheorem{them}{Theorem}

\newtheorem*{themB}{Theorem B}

\newtheorem{prop}[them]{Proposition}
\newtheorem*{propA}{Proposition A}

\newtheorem{lemma}[them]{Lemma}
\newtheorem{definition}{Definition}
\newtheorem{notation}{Notation}
\newtheorem{remark}{Remark}
\newtheorem{example}{Example}

\title
{
Asymptotically free property of the solutions \hfill\\
of  an abstract linear hyperbolic equation \hfill\\
with time-dependent coefficients \hfill
}

\author{
Taeko Yamazaki \hfill
\\
Department of Mathematics, Faculty of Science and Technology, \hfill
\\
Tokyo University of Science, Noda, Chiba, 278-8510, Japan \hfill
\\
E-mail: \lowercase{yamazaki$\_$taeko@ma.noda.tus.ac.jp} \hfill
}

\begin{document}

\maketitle

\begin{abstract}
This paper is concerned with an abstract dissipative hyperbolic equation with time-dependent coefficient. 
Under an assumption which ensures that the energy does not decay, this paper provides a condition on the coefficient, which is  necessary and sufficient so that the solutions tend to the solutions of the free wave equation. 

\bigskip
\noindent 
Keywords: {abstract linear hyperbolic equation; asymptotic behavior; asymptotically free property}
\end{abstract}

\section{Introduction}

Let $H$ be a separable complex Hilbert space $H$ with inner product $(\cdot,\cdot)_H$ and norm $\norm{\cdot}$.  
Let $A$ be a non-negative injective self-adjoint operator in $H$ with domain $D(A)$. 
Let $c(t)$ be a function which is of bounded variation and satisfies
\begin{equation}\label{cass1}
\inf_{t \ge 0} c(t) > 0.
\end{equation}
We consider the initial value problem of the abstract dissipative wave equation
\begin{align}\label{WE}
& u^{\prime \prime}(t) + c(t)^2 Au(t) + b(t) u^\prime(t) = 0 \quad t \ge 0, 
\\
&u(0) = \phi_0, \quad u^\prime(0) = \psi_0.
\label{initial}
\end{align}
with time-dependent coefficients. 
There are a number of results concerning \eqref{WE}--\eqref{initial} 
(see, for example, \cite{A, Ma, MR, RW1, RW2, EFH}, \cite[Section 2]{YaHP} and references therein). 

In this paper, under the assumption that $b(t)$ is an integrable function on $[0,\infty)$, 
we give a necessary and sufficient condition for the existence of a wave speed $c_*$ and a solution $v$ of the free wave equation 
\begin{align}\label{fe}
&v^{\prime \prime}(t) + c_*^2 Av(t) = 0, 
\end{align}
satisfying 
\begin{equation}\label{limit0}
\lim_{t \to \infty} \left( \norm{A^{1/2}(u(t) - v(t))} + \norm{u^\prime(t) - v^\prime(t)} \right) = 0.
\end{equation}

First, Arosio \cite[Theorem 3]{A} considered 
\begin{alignat}{2} 
	\label{wave1}
	&\pd{^2 u}{t^2}(t,x)  = a(t) \Delta u(t,x) + G(x,t) + H(x,t)
	&&
	\text{ in  }[0,\infty) \times \Omega,
	\\
	\label{wave2}
	& u(t,x) = 0 
	&& 	\text{ on } [0,\infty) \times \partial \Omega,
	\\
	\label{wave3}
	&u(0,x) = \phi_0(x), 
	\quad \pd{u}{t}(0,x) = \psi_0(x) \; 
	&& 	\text{ in } \Omega.
\end{alignat}
for  a  bounded open set $\Omega$ in $\R^n$,    
where $a(t) = c(t)^2 + d(t)$ with $c(t)^2 \in BV(0,\infty)$ and $d(t) \in L^1(0,\infty)$ satisfying  
	$0 <  \nu \le a(t)$ for almost every $t \in (0,\infty)$, $G \in L^1((0,\infty);L^2(\Omega))$, and 
	$H  \in BV((0,\infty); H^{-1}(\Omega))$ with $\lim_{t \to \infty} H(t) = 0 $ in $H^{-1}(\Omega)$. 
Then he showed the following. 
\begin{enumerate}
		\item If 
\begin{equation}\label{c-cinf}
\lim_{t \to \infty} \int_0^t (c(s) - c_\infty) ds 
\; \text{exists and is finite},  
\end{equation}
where
\begin{equation*}
	c_\infty = \lim_{t \to \infty} c(t), 
	\end{equation*}
then for every weak solution $u \in C([0,\infty); H^1_0(\Omega)) \cap  C^1([0,\infty);  L^2(\Omega))$ of 
\eqref{wave1}--\eqref{wave2},  
there exists a solution 
$v \in C([0,\infty); H^1_0(\Omega)) \cap  C^1([0,\infty);  L^2(\Omega))$ of  the free wave equation 
\begin{alignat}{2} 
\label{fwave1}
&\pd{^2 v}{t^2}(t,x)  = c_\infty^2 \Delta v(t,x) 
&&
\text{ in  }[0,\infty) \times \Omega,
\\
\label{fwave2}
& v(t,x) = 0 
&&
\text{ on } [0,\infty) \times \partial \Omega,
\end{alignat}
satisfying
\begin{equation}
\label{limit1}
\lim_{t \to \infty} \left( \Norm{u(t) - v(t)}{{H^1_0(\Omega)}} + \Norm{u^\prime(t) - v^\prime(t)}{{L^2(\Omega)}} 
\right) = 0.
\end{equation}

\item Conversely, if there exists a weak solution $u(t) \in C([0,\infty); H^1_0(\Omega)) \cap  C^1([0,\infty);  L^2(\Omega))$ of 
\eqref{wave1}--\eqref{wave2} 
and a non-trivial solution $v(t)$ of the free wave equation 
\eqref{fwave1}--\eqref{fwave2}  
such that \eqref{limit1} holds, 
then \eqref{c-cinf} must hold. 
\end{enumerate}
If we take $H = L^{2}(\Omega)$, $A = - \Delta$ with $D(A) = H^2(\Omega) \cap H^1_0(\Omega)$ and 
$b(t) \equiv 0$, 
the abstract problem \eqref{WE}--\eqref{initial} becomes \eqref{wave1}--\eqref{wave2} above with 
$a(t) = c(t)^2$  and  $G(t) \equiv H(t)  \equiv 0$. 
The method of \cite{A} is applicable 
for  positive self-adjoint operators $A$ with compact resolvent. 
Here we note that  if $c(t)$ satisfies \eqref{cass1},  
then the assumptions $c^2(t) \in BV([0,\infty))$ and $c(t) \in BV([0,\infty))$ are equivalent.
 
Matsuyama \cite[Theorem 2.1]{Ma} considered the problem \eqref{wave1}--\eqref{wave2} for $\Omega = \R^n$, 
where $a(t) = c(t)^2$ with $c(t)$ satisfying  \eqref{cass1} and 
\begin{equation}
\label{cassM}
c \in Lip_{loc}([0,\infty)), \quad c^\prime \in L^1(0,\infty),  
\end{equation}
$G(t) \equiv H(t) \equiv 0$, that is, the problem   
\eqref{WE}--\eqref{initial} with $H = L^2(\R^n)$, $A = -\Delta$ with $D(A) = H^2(\R^n)$ and $b(t) \equiv 0$, 
and showed the following:  
Assume that  \eqref{c-cinf} holds. 
Then for every solution $u \in \bigcap_{j=0,1,2} C^j([0,\infty); {H}^{s - j}(\R^n))$ $(s \ge 1)$ of \eqref{wave1}--\eqref{wave2}, 
there exists a solution $v$ 
 of  the free wave equation \eqref{fwave1}--\eqref{fwave2} satisfying
\begin{equation}\label{uvRn}
\lim_{t \to \infty} \left( \Norm{\nabla(u(t) - v(t))}{{H^{s-1}(\R^n)}}
+ \Norm{u^\prime(t) - v^\prime(t)} {{H^{s-1}(\R^n)}} \right) = 0.
\end{equation}
On the other hand, he showed that 
if
\begin{equation}\label{cabs}
\lim_{t \to \infty} \left| \int_0^t (c(s) - c_\infty) ds \right| = \infty,
\end{equation}
there exists a non-trivial free solution $u$ of  \eqref{wave1}--\eqref{wave2} 
such that no solution $v$ of the free wave equation \eqref{fwave1}--\eqref{fwave2} satisfies \eqref{uvRn}. 
Then,  applying the result to Kirchhoff equation, he proved in \cite{Ma_exam}  the existence of a non-trivial small initial data  such that the solution of Kirchoff equation is not asymptotically free. 

 Matsuyama and Ruzhansky \cite[Theorem 1.1]{MR} considered the system  
$D_t U = A(t,D_x)U $ in $L^2(\R^n)^m$, and generalized the results of \cite{Ma}. Furthermore, in a case $m = 1$ and $A(t,D_x) = - c(t)^2 \Delta$, this result is an improvement of the necessary condition for the asymptotically freeness of  \cite{Ma} as follows: 
Assume that $c$ satisfies \eqref{cass1} and \eqref{cassM}. 
If \eqref{cabs} holds, then for every non-trivial solutions of \eqref{wave1}--\eqref{wave2} with radially symmetric initial data, there exists no solution of the free wave equation  \eqref{fwave1}--\eqref{fwave2} satisfying \eqref{uvRn}. 

The purpose of this paper is to show a necessary and sufficient condition for asymptotically free property of \eqref{WE}--\eqref{initial} 
for general non-negative injective self-adjoint operator $A$ (Theorem 1).  
Especially we are interested in the necessary condition. 
To obtain the necessary condition, Arosio \cite[Theorem 3, (\rom 2)]{A} employed      
the discreteness of the spectrum corresponding to $A$,   
and Matsuyama and  Ruzhansky \cite [Theorem 1.1]{MR} employed  the Riemann--Lebesgue theorem for the Fourier transform. 
In this paper, we use the property of continuous unitary group $e^{i t A^{1/2}}$.  

Another difference between the previous results and the result of this paper is that we do not assume $c_* = c_\infty$ in \eqref{fe} a priori. 
We show that  if there exists a non-trivial solution $u$ of \eqref{WE}  
which approaches to a solution of  \eqref{fe} with some wave speed $c_*$, 
then $c_*$ coincides with $c_\infty= \lim_{t \to \infty} c(t)$ 
(Theorem 1 (\rom 2)). 

 The result of this paper is applied to dissipative Kirchhoff equations in \cite{Ya}  to obtain the necessary decay  condition on the dissipative term for the asymptotically free property. 
This condition is  essentially  stronger than that of linear dissipative wave equation.  

\section{Main result}

\begin{notation}
For every $\alpha \ge 0$,   the domain $D(A^{\alpha})$ of $A^{\alpha}$  
becomes a Hilbert space $\mathbb{H}_\alpha$ equipped with the inner product 
$$
(f,g)_{{\mathbb{H}_\alpha}} := (A^{\alpha}f, A^{\alpha}g)_H + (f,g)_H. 
$$
The norm is denoted by $\Norms{f}{{\mathbb{H}_\alpha}} = (f,f)_{{\mathbb{H}_\alpha}}$. 
We note that $\HHH_0 = H$. 
For every $\alpha < 0$, 
let $\mathbb{H}_\alpha$ denote the dual space of $\mathbb{H}_{-\alpha}$ with the dual norm, namely, 
$\HHH_\alpha$ is the completion of $H$ by the norm 
$$
 \Norm{f}{{\mathbb{H}_\alpha}} 
= \sup \{ |(f,g)_{H}| ; g \in \mathbb{H}_{-\alpha},  \Norm{g}{{-\alpha}} = 1 \}.
$$
\end{notation}

\begin{notation}
For every $\alpha > 0$, let $\HH_\alpha$ denote the completion of $D(A^\alpha)$ by the norm $\norm{A^\alpha \cdot}$. 
Let  $\A^\alpha$ be extension of $A^\alpha$ on $\HH_\alpha$. 
The fact that  $A^\alpha$ is an injective self-adjoint operator implies that the range $R(A^\alpha)$ is dense in $H$,  
and thus $\A^\alpha: \HH_\alpha \to H$ is bijective. 
From this fact and the definition, it follows that
$
\A^\alpha: \HH_\alpha \to H
$ 
is an isometric isomorphism. 
\end{notation}

\begin{example}
Let $H = L^2(\R^n)$ and $A = -\Delta$ with $D(A) = H^2(\R^n)$. 
For $\alpha > 0$, the space $\HH_\alpha(\R^n) $ equals the homogeneous Sobolev space ${\dot H}^{2 \alpha}$, and $\mathbb{H}_{-\alpha}$ equals the negative Sobolev space $H^{-2\alpha}(\R^n)$.   
\end{example} 
\bigskip

\begin{notation}
For a Banach space $X$, let $AC([0,\infty);X)$  denote all of $X$ valued absolutely continuous functions on $[0,\infty)$, and 
$AC_{\loc}([0,\infty);X) = \{f \in C([0,\infty)); f \in AC([0,T]) $ 
for every $T > 0 \}$.
\end{notation}
\textbf{}
\medskip
We consider the equation \eqref{WE}--\eqref{initial} and a free wave equation \eqref{fe} in a somewhat wide class as 
\begin{align}\label{we0}
&u^{\prime \prime}(t) + c(t)^2 A^{1/2} \A^{1/2} u(t) + b(t) u^\prime(t) = 0, \quad t \ge 0, 
\\
&u(0) = \phi_0, \quad u^\prime(0) = \psi_0,  
\label{initial0}
\end{align}
for $(\phi_0, \psi_0) \in \HH_{1/2} \times H$, and 
\begin{align}\label{fe0}
&v^{\prime \prime}(t) + c_*^2 A^{1/2} \A^{1/2} v(t)  = 0, \quad t \ge 0.  
\end{align}
 
\begin{definition} We say that $u$ is a weak solution of \eqref{we0} if  
$u \in C \left([0,\infty): \HH_{1/2} \right)$, 
\begin{equation*}
\begin{aligned}
u(t) - u(0) & \in \bigcap_{j = 0,1} C^j ([0,\infty);\mathbb{H}_{(1-j)/2} ),
\\
u^\prime(t) &\in  AC_{\loc} \left([0,\infty); \HHH_{-1/2} \right),
\end{aligned}
\end{equation*} 
and \eqref{we0} holds in the space $\mathbb{H}_{-1/2}$ for almost every 
$t \in (0,\infty)$.  

A weak solution of \eqref{fe0} is defined as a weak solution of  \eqref{we0} with $c(t) \equiv c^*$ and $b(t) \equiv 0$.  
\end{definition}

Here we note that if $u$ is a weak solution of \eqref{we0}--\eqref{initial0}, then $\mathbf{x} : = (\A^{1/2}u, u^\prime)$ 
is a weak solution of the following Cauchy problem:
\begin{align}
\label{WES}
&\dfrac{d}{dt} \mathbf{x}(t)
+ \begin{pmatrix} 
0     & - A^{1/2} \\
c(t)^2 A^{1/2} & b(t)   
\end{pmatrix}
\mathbf{x}(t)
=  
\begin{pmatrix}
 0 \\
 0 
\end{pmatrix},
\\
\label{initialS}
&\mathbf{x}(0) = \begin{pmatrix}
\A^{1/2}\phi_0 \\
 \psi_0 
\end{pmatrix} \in H \times H, 
\end{align}
in the sense that 
$$
\mathbf{x}(t)
\in C \left([0,\infty); H \times H \right) \cap
AC_{\loc} \left([0,\infty); \HHH_{-1/2} \times \HHH_{-1/2} \right), 
$$
and that \eqref{WES} holds in $\HHH_{-1/2} \times \HHH_{-1/2}$ for almost every  $t \in (0,\infty)$. 
Conversely, if 
$\mathbf{x} = (w, z) $  
is a weak solution of \eqref{WES}--\eqref{initialS},  
then  $u = \A^{-1/2}w$ is a weak solution of \eqref{we0}--\eqref{initial0}.

Our main result is the following:

\begin{them}\label{maintheorem}
Let $c(t)$ be of bounded variation on $(0,\infty)$  
 satisfying \eqref{cass1}, and put $c_\infty = \lim_{t \to \infty} c(t)$.  
Let $b(t)$ be an integrable function on $[0, \infty)$. 
Then the following holds. 
\begin{enumerate}
\item
Suppose that \eqref{c-cinf} holds. 
Then  
for every weak solution $u$ of \eqref{we0}, 
there exists a unique weak solution $v$ of the free wave equation \eqref{fe0} with wave speed $c_* = c_\infty$ 
such that 
\begin{equation}\label{limit}
	\lim_{t \to \infty} \left( \norm{\A^{1/2}(u(t) - v(t))} + \norm{u^\prime(t) - v^\prime(t)} \right) = 0
\end{equation}
holds. 

\item Suppose that there exists a non-trivial weak solution $u$ of \eqref{we0},  
a positive constant $c_*$ and a weak solution $v$ of the free wave equation \eqref{fe0}  
such that \eqref{limit} holds.   
Then $c_* = c_\infty$ and  \eqref{c-cinf} must hold. 
\end{enumerate}
\end{them}

\begin{remark}
	If $b(t)$ is integrable and of bounded variation as well, the Cauchy problem \eqref{we0}--\eqref{initial0} is uniquely solvable. (See Proposition \ref{prop:global1} in Appendix.)
\end{remark}

\begin{remark}
Assume that the initial data $(\A^{1/2} \phi_0, \psi_0)$ belongs to $D(A^{J/2}) \times D(A^{J/2}) $ for $J \ge 1$, 
and  $u$ is a solution of \eqref{we0}--\eqref{initial0} in the sense that  
\begin{equation}\label{Jsol}
(\A^{1/2} u,   u^\prime) 
\in C \left([0,\infty); \HHH_{J} \times \HHH_J \right) \cap
AC_{\loc} \left([0,\infty); \HHH_{J-1/2} \times \HHH_{J-1/2} \right), 
\end{equation}
and that \eqref{WES} holds in $\HHH_{(J-1)/2} \times \HHH_{(J-1)/2}$ for almost every  $t \in (0,\infty)$. 
Then the solution $v$ of \eqref{fe0} given by (\rom 1) of Theorem 1 satisfies \eqref{Jsol} and   
\begin{equation}\label{uvdisJ}
\lim_{t \to \infty} \left( \Norm{\A^{1/2}(u(t) - v(t))}{{\HHH_{J/2}}} + \Norm{u^\prime(t) - v^\prime(t)}{{\HHH_{J/2}}} \right) = 0.    
\end{equation}
In fact, since  we see that \eqref{rest} in section 2 with $\Norm{\cdot}{{H \times H}}$ replaced by $\Norm{\cdot}{{\HHH_{J/2}  \times \HHH_{J/2} }}$ holds, 
we can prove \eqref{uvdisJ} in the same way as in the proof of Theorem 1 (\rom 1).    
\end{remark}

\section{Proof of Theorem \ref{maintheorem}}

We first give a lemma, which is employed in the proof of the equality $c_* = c_\infty$. 

\begin{lemma}\label{lem2}
	If $g(t)$ is of bounded variation on $[0,\infty)$, then 
	\[
	\lim_{T\to\infty}\left\|\frac{1}{T}\int_0^T g(t) \exp(iG(t) A^{1/2})u \,dt\right\| = 0
	\]
	for every $u \in H$, where $G(t) = \int_0^t g(s) ds$. 	
\end{lemma}

\begin{proof}
Let $w $ be an arbitrary element of $D(A^{1/2})$.  Then,  $\exp(iG(t)A^{1/2})w$ is absolutely continuous on $[0,\infty)$ and differentiable almost everywhere on $(0,\infty)$, and thus we have
\begin{equation}\label{sg}
\frac{d}{dt}\exp(iG(t) A^{1/2})w = i g(t) \exp(iG(t)A^{1/2})A^{1/2}w.
\end{equation}
for almost every $t$ in $(0,\infty)$. 
Integrating \eqref{sg} on $(0,T)$, and dividing the equality by $T$, we have 
\[
\frac{1}{T}\int_0^T \exp(iG(t)A^{1/2})g(t)A^{1/2}w\,dt = \frac{\exp(i G(T) A^{1/2})w-w}{iT}.
\]
Since $\|\exp(i\tau(T)A^{1/2})w\| = \|w\|$ for every $T \in [0,\infty)$, 
we obtain
\[
\lim_{T\to\infty}\left\|\frac{1}{T}
\int_0^T \exp(iG(t)A^{1/2})g(t) A^{1/2}w\,dt\right\| = 0.
\]
Let $\delta > 0$ be an arbitrary positive number. 
The assumption that $A$ is an injective self-adjoint operator implies that the range of $A^{1/2}$ is dense in $H$. 
Thus, we can take   
$w \in D(A^{1/2})$ such that $\|u-A^{1/2}w\| < \delta$, 
and therefore we have 
\begin{align*}
&    \limsup_{T\to\infty}\left\|\frac{1}{T}\int_0^Tg(t) \exp(iG(t)A^{1/2}) u \,dt \right\|
\\
&\le \limsup_{T\to\infty}\left\|\frac{1}{T}\int_0^Tg(t) \exp(iG(t)A^{1/2})(u-A^{1/2}w)\,dt
\right\|
\\ & \qquad
+ \lim_{T\to\infty}\left\|\frac{1}{T}\int_0^T g(t) \exp(iG(t) A^{1/2})A^{1/2}w\,dt\right\|
\\
&\le \sup_{t \ge 0}\left(|g(t)| \big\|\exp(iG(t) A^{1/2})(u-A^{1/2}w)\big\|
\right)
\\
&\le  \delta \sup_{t \ge 0}|g(t)|.
\end{align*}
Since $\delta > 0$ is arbitrary, we obtain
\[
\lim_{T\to\infty}\left\|\frac{1}{T}\int_0^T  g(t) 
\exp(iG(t) A^{1/2})u\,dt \right\| = 0.
\]
\end{proof}

Now we prove Theorem 1. 
We express the solution $\mathbf{x}(t)$ of \eqref{WES} by the method of ordinary differential equation by  Wintner \cite{Wi} (see also Coddington and Levinson \cite{CL}, Hartman \cite{Ha}),   
similarly to the proof of Matsuyama \cite{Ma}.  
  Let
\begin{equation*}
Y(t):= 
\begin{pmatrix}
  e^{i \tau(t) A^{1/2}}                            
   &  e^{-i \tau(t)  A^{1/2} } \\
   i c(t)  e^{i \tau(t) A^{1/2} } &   - i c(t) e^{-i \tau(t) A^{1/2} }
\end{pmatrix},
\end{equation*}
where
\begin{equation*}
\tau(t) = \int_0^t c(s) ds.
\end{equation*}
Then
\begin{equation*}
Y(t)^{-1}= \frac{1}{2}
\begin{pmatrix}
  e^{-i \tau(t) A^{1/2}}  &  - \frac{i}{ c(t) } e^{-i \tau(t)  A^{1/2} } \\
 e^{i \tau(t)  A^{1/2}}  &    \frac{i}{ c(t) }  e^{i \tau(t) A^{1/2} } 
\end{pmatrix}.
\end{equation*}

In order to approximate $c$ by $C^1$ class functions, we use the mollifier as in the proof of Arosio \cite{A}. 
Let $\rho$ be a $C_0^\infty(\R)$ function with support contained in $[-1,1]$ and 
$\int_\R \rho(t) dt = 1$. 
Let $\delta$ be an arbitrary positive number. Put
$
\rho_\delta = \frac{1}{\delta}\rho(\frac{t}{\delta}),
$
and $c_\delta$ be the mollification of $c$, that is, 
$$
c_\delta(t) = \tilde{c} * \rho_\delta(t) = \int_{\R} \tilde{c}(t-s)\rho_\delta(s)ds (\in C^\infty(\R)),
$$
where $\tilde{c}$ is a extension of $c$ to $\R$ such that 
$\tilde{c}(t) = c(0)$ for $t < 0$. 
From the assumption that $c$ is bounded variation on $[0,\infty)$,  it follows that  
\begin{align}\label{c-cdelta}
\int_{S}^T |c(s) - c_\delta(s)|ds & \le 
\delta \Var(c;[\max \{S - \delta, 0 \},T+ \delta]),
\\
\label{cdelta_der}
\int_S^T |c_\delta^\prime(s)| ds & \le 
\Var(c;[\max \{S - \delta, 0 \},T+ \delta]),
\end{align}
for every $S,T \ge 0$ with $S < T$ (see \cite{CGS} and \cite{A}).  
Inequality \eqref{c-cdelta} with $\delta = 1/n$ implies  $\lim_{n \to \infty} c_{1/n} = c$ in $L^1((0,\infty))$. 
Thus, we can take a subsequence $\{ n_k \}_{k = 1}^\infty$ and a subset $N_1 \subset (0,\infty)$ such that the Lebesgue measure of $N_1$ is $0$ and that  
\begin{equation}\label{cn}
\lim_{k \to \infty} c_{1/n_k}(t) = c(t)
\end{equation}
for every $t \in (0,\infty) \setminus N_1$.   
Let
\begin{equation*}
Y_k(t):= 
\begin{pmatrix}
e^{i \tau(t) A^{1/2}}                             &  e^{-i \tau(t)  A^{1/2} } \\
i c_{{1/n_k}}(t)  e^{i \tau(t) A^{1/2} } &   - i c_{{1/n_k}}(t) e^{-i \tau(t) A^{1/2} }
\end{pmatrix}.
\end{equation*}
Then
\begin{equation*}
Y_k(t)^{-1}= \frac{1}{2}
\begin{pmatrix}
e^{-i \tau(t) A^{1/2}}  &  - \frac{i}{c_{{1/n_k}}(t) } e^{-i \tau(t)  A^{1/2} } \\
e^{i \tau(t)  A^{1/2}}  &    \frac{i}{c_{{1/n_k}}(t)} e^{i \tau(t)  A^{1/2} } 
\end{pmatrix},
\end{equation*}
and 
\begin{equation*}
\dfrac{d}{dt} Y_k(t)
+ \begin{pmatrix} 
0                          & - \frac{c(t)}{c_{{1/n_k}}(t)} A^{1/2} \\
c(t)c_{{1/n_k}}(t) A^{1/2} & -\frac{c_{{1/n_k}}^\prime(t)}{c_{{1/n_k}}(t)}    
\end{pmatrix}
Y_k(t)
=  
\begin{pmatrix} 
0    & 0 \\
0   & 0
\end{pmatrix}.
\end{equation*}
From \eqref{cass1} and \eqref{cn}, it follows that
\begin{align}\label{Ynlimit}
&\lim_{k \to \infty}\Norm{Y_k (t)^{-1} - Y(t)^{-1}}{{\LL(H \times H)}} = 0 
\; &\text{for every } \; t \in (0,\infty) \setminus N_1. 
\end{align} 

Let $\mathbf{x}(t) $ be a weak solution of \eqref{WES}--\eqref{initialS}. 
By putting
\begin{align*}
B_k(t) 
&=  Y_k(t)^{-1} 
\begin{pmatrix} 
0                          & \left(\frac{c(t)}{c_{{1/n_k}}(t)}- 1 \right) \A^{1/2} \\
c(t) (c(t) - c_{{1/n_k}}(t))\A^{1/2} &  b(t) + \frac{c_{{1/n_k}}^\prime(t)}{c_{{1/n_k}}(t)^2}    
\end{pmatrix}
Y_k(t), 
\end{align*}
$$
\mathbf{y}(t) = \begin{pmatrix}
 y_1(t) \\
 y_2(t)                                                                                              
\end{pmatrix}
:= Y(t)^{-1} \mathbf{x}(t),
$$
and
$$
\mathbf{y}_k(t) = \begin{pmatrix}
y_k^{(1)}(t) \\
y_k^{(2)}(t)                                                                                     
\end{pmatrix}
:= Y_k(t)^{-1} \mathbf{x}(t),
$$
\eqref{WES} is transformed into 
\begin{equation}
\label{WEST}
\dfrac{d}{dt} \mathbf{y}_k(t) + B_k(t) \mathbf{y}_k(t) = 0
\quad \text{in} \; \HHH_{-1/2} \times \HHH_{-1/2}. 
\end{equation}

Let $\{E(\lambda)\}$ be a spectral family associated with the self adjoint operator $A$. 
Then \eqref{WEST} yields
\begin{equation}\label{ylambdadelta_eq}
\dfrac{d}{dt} (E(\lambda)\mathbf{y}_k)(t) + B_{\lambda,k}(t) (E(\lambda)\mathbf{y}_k)(t) = 0  \;\; \text{in} \; H \times H, 
\end{equation}
for almost every $t \in  (0,\infty)$, where
\begin{align*}
&B_{\lambda,k}(t) 
\\
&=  Y_k(t)^{-1} 
\begin{pmatrix} 
0                                             
& \left(\frac{c(t)}{c_{{1/n_k}}(t)}- 1 \right) \A^{1/2}E(\lambda) 
\\
c(t) (c(t) - c_{{1/n_k}}(t))\A^{1/2}E(\lambda) 
& b(t)+\frac{c_{{1/n_k}}^\prime(t)}{c_{{1/n_k}}(t)^2}    
\end{pmatrix}
 Y_k(t).
\end{align*}
By \eqref{cass1} and the fact that $e^{\pm i s A^{1/2}}$ is unitary, 
the operators $Y_k(t)$ and $Y_k(t)^{-1}$ are bounded on $H \times H$ uniformly in $k$ and $t$.  
Thus, observing \eqref{cass1} again, we have a positive constant $K_1$ satisfying
\begin{equation}\label{Blambdadelta}
\Norm{B_{\lambda,k}(t)}{{\LL(H \times H)}} \le K_1 \left(  
 \lambda^{1/2} \left|c(t) - c_{{1/n_k}}(t) \right|
+ |c_{{1/n_k}}^\prime(t) |+ |b(t) |
\right)
\end{equation}
for every $\lambda, k > 0$ and every $t \ge 0$. 

We estimate  $(E(\lambda)\mathbf{y}_k)(t)$. 
The definition of weak solution implies $\mathbf{x}(t) \in AC_{\loc}([0,\infty); \HHH_{-1/2} \times \HHH_{-1/2})$, and therefore, $(E(\lambda)\mathbf{y}_k)(t) \in AC_{\loc}([0,\infty); H \times H)$.  
Thus, it follows from \eqref{ylambdadelta_eq} and \eqref{Blambdadelta} that 
\begin{equation}\label{Yts}
\begin{aligned}
&\Norm{(E(\lambda)\mathbf{y}_k)(t) - (E(\lambda)\mathbf{y}_k)(s)}
{{H \times H}} 
\\
&\le K_1 \int_s^t 
\left(
\lambda^{1/2} \left|c(\sigma) - c_{{1/n_k}}(\sigma) \right|  
+ |c_{{1/n_k}}^\prime(\sigma) |+ |b(\sigma) |
\right)
\Norm{(E(\lambda)\mathbf{y}_k)(\sigma) }{{H \times H}} d \sigma,
\end{aligned}
\end{equation}
for every $0 < s < t$. 
Thus
\begin{equation*}\begin{aligned}
&\Norm{(E(\lambda)\mathbf{y}_k)(t) }{{H \times H}} 
\le  
\Norm{(E(\lambda)\mathbf{y}_k)(0)}{{H \times H}} 
\\
&+  K_1 \int_0^t 
\left(
\lambda^{1/2}| c(\sigma) - c_{{1/n_k}}(\sigma) | + |c_{{1/n_k}}^\prime(\sigma) |+ |b(\sigma) |
\right)
\Norm{(E(\lambda)\mathbf{y}_k)(\sigma) }{{H \times H}} d \sigma,
\end{aligned}
\end{equation*}
for every $t \ge 0$. 
Hence by Gronwall's inequality together with the assumption that $b \in L^1((0,\infty))$, \eqref{c-cdelta} and \eqref{cdelta_der}, 
\begin{equation*}\begin{aligned}
&\Norm{(E(\lambda)\mathbf{y}_k)(t)}{{H \times H}} 
\\
&\le 
\exp\left(K_1 \left( ( \lambda^{1/2}/n_k + 1) \Var(c;[0,\infty)) + \Norm{b}{{L^1(0,\infty)}} \right) \right)
\Norm{(E(\lambda)\mathbf{y}_k)(0)}{{H \times H}}
\\
&\le 
 \exp\left(K_1 \left( (\lambda^{1/2}/n_k+ 1) \Var(c;[0,\infty)) + \Norm{b}{{L^1(0,\infty)}} \right) \right)
\Norm{\mathbf{y_k}(0)}{{H \times H}}
\\
\end{aligned}\end{equation*}
for every $t \ge 0$.   
Substituting this inequality into \eqref{Yts}, and observing  \eqref{c-cdelta} and \eqref{cdelta_der} again, we obtain
\begin{equation}\label{ylambdadeltats}
\begin{aligned}
&  \Norm{(E(\lambda)\mathbf{y}_k)(t) - (E(\lambda)\mathbf{y}_k)(s)}
{{H \times H}} 
\\
&\le 
K_1 
\left (  (  \lambda^{1/2}/n_k+ 1) \Var(c;[\max \{s-(1/n_k), 0 \}, t+(1/n_k)]) 
+ \Norm{b}{{L^1(s,t)}} \right)
\\
& \quad \times  
\exp\left(K_1 \left((\lambda^{1/2}/n_k + 1) \Var(c;[0,\infty)) + \Norm{b}{{L^1(0,\infty)}} \right)\right)
\Norm{\mathbf{y_k}(0)}{{H \times H}} 
\end{aligned}
\end{equation}
for every $0 \le s \le t$. 
From \eqref{Ynlimit}, it follows that
\begin{align*}
& \lim_{k \to \infty}\Norm{\mathbf{y}_{k}(t) - \mathbf{y}(t)}
{{H \times H}} = 0 
\; &\text{for every } \; t \in (0,\infty) \setminus N_1,
\end{align*}
and therefore
\begin{align*}
& \lim_{k \to \infty}\Norm{(E(\lambda)\mathbf{y}_{k})(t) - (E(\lambda)\mathbf{y})(t)}
{{H \times H}} = 0 
\quad 
\end{align*}
for every $s, t \in (0,\infty) \setminus N_1$ and $\lambda > 0$.  
Thus, letting $k \to \infty$ in \eqref{ylambdadeltats}, 
we obtain
 \begin{equation*}
\begin{aligned}
& \Norm{(E(\lambda)\mathbf{y})(t) 
 	-(E(\lambda)\mathbf{y})(s)}
 {{H \times H}} 
 \le
 K_1 \left ( \Var(c;[s-,t+]) + \Norm{b}{{L^1(s,t)}} \right)
 \\
 &  \qquad \qquad\qquad \qquad \quad \times 
  \exp\left(K_1  \Var(c;[0,\infty)) + \Norm{b}{{L^1(0,\infty)}} \right)
 \Norm{\mathbf{y}(0)}{{H \times H}} 
 \end{aligned}
 \end{equation*}
 for every $s, t \in (0,\infty) \setminus N_1$ and $\lambda > 0$, 
 where $ \Var(c;[s-,t+]) = \lim_{\delta \to  0+0}\Var(c;[s-\delta,t+\delta])$. 
Therefore we have
 \begin{equation}\label{yts}
 \begin{aligned}
 &\Norm{\mathbf{y}(t) - \mathbf{y}(s)}{{H \times H}} 
 \le K_1  \left ( \Var(c;[s-,t+]) + \Norm{b}{{L^1(s,t)}} \right)
 \\
 &  \qquad \qquad\qquad \qquad \quad \times  
 \exp\left(K_1  \Var(c;[0,\infty)) + \Norm{b}{{L^1(0,\infty)}} \right)
 \Norm{\mathbf{y}(0)}{{H \times H}} 
 \end{aligned}
 \end{equation} 
 for every $s, t \in (0,\infty) \setminus N_1$. 
 Since $c$ is of bounded variation on $[0,\infty)$, $\lim_{s,t \to \infty} \Var(c;[s-,t+]) = 0$. 
Hence, letting $s,t (\notin N_1)  \to \infty$ in \eqref{yts} implies the existence of the limit
 \begin{equation*}
 \lim_{t \notin N_1, t \to \infty}\mathbf{y}(t) 
= \mathbf{y}_\infty 
:= \begin{pmatrix} y_{1,\infty}\\
y_{2,\infty} \end{pmatrix}
\text{ in } H \times H.  
 \end{equation*}
Thus $\mathbf{y}(t)$ is expressed as 
\begin{equation*}
\mathbf{y}(t) = \mathbf{y}_\infty + \mathbf{r}(t),
\end{equation*}
with
\begin{equation}\label{rest}
\lim_{t \notin N_1, t \to \infty}\Norm{\mathbf{r}(t)}{{H \times H}} = 0.
\end{equation}
Hence we obtain the expression of the solution of \eqref{WE}
\begin{align}\label{uexp1}
\begin{pmatrix}
\A^{1/2}u(t) \\
        u^\prime(t)
\end{pmatrix}
& = \mathbf{x}(t) = Y(t) \mathbf{y}(t) 
=  Y(t) \mathbf{y}_\infty +  Y(t) \mathbf{r}(t)
\\
& = 
\begin{pmatrix}
  e^{i \tau(t) A^{1/2}} y_{1,\infty} + e^{-i \tau(t)  A^{1/2} } y_{2,\infty}  \\
   i c(t)  e^{i \tau(t) A^{1/2} } y_{1,\infty}   - i c(t) e^{-i \tau(t) A^{1/2}} y_{2,\infty}
\end{pmatrix}
+ Y(t) \rr(t).
\label{uexp}
\end{align}

Let $v$ be a solution of \eqref{fe}. 
Then it is expressed as 
\begin{equation}\label{vexp}\begin{aligned}
\mathbf{z}(t) = \begin{pmatrix}
\A^{1/2} v(t)  \\
 v^\prime(t)
\end{pmatrix}
& = 
\begin{pmatrix}
  e^{i  c_* t A^{1/2}} \phi + e^{-i c_* t  A^{1/2} } \psi  \\
   i c_* e^{i c_* t A^{1/2} } \phi  - i c_*  e^{-i c_* t A^{1/2}} \psi
\end{pmatrix},
\end{aligned}
\end{equation}
where
\begin{align*}
&\phi = \frac{1}{2}\left( \A^{1/2} v(0) - \frac{i}{c_*} v^\prime(0) \right), \quad 
\psi = \frac{1}{2}\left( \A^{1/2} v(0) + \frac{i}{c_*} v^\prime(0) \right).
\end{align*}
Since $\mathbf{x},\mathbf{z} \in C([0,\infty);H \times H)$, we easily see that 
$\lim_{t \to \infty} \Norm{\mathbf{x}(t) -\mathbf{z}(t)}{{H \times H}} = 0$ if and only if 
\begin{equation}\label{tGinfty}
\lim_{t \notin N_1, t \to \infty} 
\Norm{\mathbf{x}(t) -\mathbf{z}(t)}{{H \times H}} = 0. 
\end{equation}
Thus, the convergence \eqref{limit} holds if and only if \eqref{tGinfty} holds.  
By the expressions \eqref{uexp} and \eqref{vexp}, we see that \eqref{tGinfty} holds if and only if the following two convergences hold. 
\begin{align}\label{diss1}
&\lim_{t \notin N_1, t \to \infty} \norm{ e^{i \tau(t) A^{1/2}}  y_{1,\infty} + e^{-i \tau(t)  A^{1/2} } y_{2,\infty} 
-   e^{i c_* t A^{1/2}}  \phi   - e^{-i c_* t  A^{1/2} } \psi   } = 0,
\\
& \lim_{t \notin N_1, t \to \infty}  \norm{  c(t)  e^{i \tau(t) A^{1/2} } y_{1,\infty}   -  c(t) e^{-i \tau(t) A^{1/2}} y_{2,\infty}
\nonumber \\
&\qquad  \qquad \qquad \qquad 
-    c_* e^{i c_* t A^{1/2} } \phi  + c_*  e^{-i c_* t A^{1/2}} \psi }
= 0.
\label{diss2}
\end{align}

Here we prove the following lemma. 
\begin{lemma}\label{equiv}
Assume that $v$ is a weak solution of linear wave equation of \eqref{fe} with 
		\begin{equation}\label{cinfdef}
	c_* = c_\infty (= \lim_{t \to \infty} c(t)).
	\end{equation}
	Then the  convergence \eqref{limit} holds if and only if the following two convergences hold: 
	\begin{align}
		& \lim_{t \notin N_1, t \to \infty}  \norm{e^{i (\tau(t) - c_\infty t) A^{1/2}} y_{1,\infty} -    \phi  }
	= 0,
	\label{y1limit}
	\\
	& \lim_{t \notin N_1, t \to \infty} \norm{e^{-i (\tau(t) - c_\infty t) A^{1/2}}  y_{2,\infty} -  \psi  } 
	=  0. 
	\label{y2limit}
	\end{align}
	\end{lemma}
	
	\begin{proof}
		By the argument above, the convergence \eqref{limit} holds if and only if \eqref{diss1} and \eqref{diss2} hold. 		
By the assumption \eqref{cinfdef} and the fact that $e^{i \tau(t) A^{1/2}}$ is a $C^0$ unitary group on $H$, we see that \eqref{diss2} holds if and only if the following convergence holds. 
\begin{align}
& \lim_{t \notin N_1, t \to \infty}  \norm{   e^{i \tau(t) A^{1/2} } y_{1,\infty}   -   e^{-i \tau(t) A^{1/2}} y_{2,\infty}
	-    e^{i c_\infty t A^{1/2} } \phi  +  e^{-i c_\infty t A^{1/2}} \psi }
= 0.
\label{diss3}
\end{align}
Hence, \eqref{limit} holds, if and only if \eqref{diss1} and  \eqref{diss3} hold, equivalently, the following two  convergences hold. 
\begin{align*}
& \lim_{t \notin N_1, t \to \infty} 
\norm{ e^{i \tau(t) A^{1/2}}  y_{1,\infty} 
	  	-   e^{i c_\infty t A^{1/2}}\phi  } 
= 0,
\\
&\lim_{t \notin N_1, t \to \infty} \norm{ e^{-i \tau(t) A^{1/2}}  y_{2,\infty} -   e^{-i c_\infty t A^{1/2}} \psi  } 
=  0. 
\end{align*}
Since $e^{i s A^{1/2}}$ is a unitary group on $H$, these convergences are equivalent to \eqref{y1limit} and \eqref{y2limit}. 
\end{proof}

Now we are ready to complete the proof of Theorem 1. 

{\bf Proof of (\rom 1)}. 
Assume that \eqref{c-cinf} holds. 
We take 
$c_*= c_\infty (= \lim_{t \to \infty} c(t))$,   
and
$$
\phi = e^{i \lim_{t \to \infty} (\tau(t) - c_\infty t) A^{1/2}} y_{1,\infty}, \quad 
\psi = e^{-i \lim_{t \to \infty} (\tau(t) - c_\infty t) A^{1/2}} y_{2,\infty}.
$$
Then by the strong continuity of the $e^{i s A^{1/2}}$ with respect to $s$ on $[0,\infty)$, the convergences \eqref{y1limit} and \eqref{y2limit} hold, and therefore \eqref{limit} holds by Lemma \ref{equiv}.  

{\bf Proof of (\rom 2)}.
Assume that there are a non-trivial solution $u$ of \eqref{we0},  a positive number $c_*$ and a  solution $v$ of \eqref{fe0} such that \eqref{limit} holds. 
Put
$$ 
F(t) := \frac{1}{2}\norms{u^\prime(t)} + \frac{1}{2}c(t)^2 \norms{\A^{1/2} u(t)}
$$
for every $t \ge 0$. 
Since $u$ is non-trivial and $\norms{u^\prime(t)} +  \norms{\A^{1/2} u(t)}$ is continuous, there is $S \in [0,\infty) \setminus N_1$ such that $\norms{u^\prime(S)} +  \norms{\A^{1/2} u(S)} > 0$.  Then by \eqref{cass1}, we have 
\begin{equation}\label{ES}
F(S) > 0. 
\end{equation}
For every $\lambda > 0$ and $\delta > 0$, we put $u_\lambda = E([0,\lambda))u$, 
\begin{align*}
F_{\lambda}(t) &:= \frac{1}{2}\norms{u_\lambda^\prime(t)} 
+ \frac{1}{2}c(t)^2 \norms{\A^{1/2} u_\lambda(t)}
\quad \text{for every}\;t \ge 0, 
\\
F_{\lambda, k}(t) &:= \frac{1}{2}\norms{u_\lambda^\prime(t)} 
+ \frac{1}{2}c_{{1/n_k}}(t)^2 \norms{\A^{1/2} u_\lambda(t)} 
\quad \text{for every}\; t \ge 0.
\end{align*} 
Since $u$ satisfies \eqref{we0} in $\HHH_{-1/2} 
\times\HHH_{-1/2}$ for almost every $t \in (0,\infty) $, 
$u_\lambda$ satisfies \eqref{we0} in $H \times H$ for almost every $t \in (0,\infty)$. Thus we have
\begin{align*}
F_{\lambda, k}^\prime(t) 
&= (c_{{1/n_k}}(t)^2  - c(t)^2)(u_\lambda^\prime(t), A^{1/2}\A^{1/2}u_\lambda(t))_H 
- b(t) \norms{u_\lambda^\prime(t)} 
\\
& \qquad +  c_{{1/n_k}}(t)c_{{1/n_k}}^\prime(t)\norms{\A^{1/2} u_\lambda(t)} 
\\
&\ge - \frac{1}{c_0}|c_{{1/n_k}}(t)^2  - c(t)^2| \sqrt{\lambda} F_{\lambda, k}(t)
- 2 \left(|b(t)| + \frac{|c_{{1/n_k}}^\prime(t)|}{c_{{1/n_k}}(t)} \right) F_{\lambda, k}(t)
\\
&\ge - \left( 
\frac{2 \sqrt{\lambda}}{c_0} (\sup_{t \ge 0} c(t)) |c_{{1/n_k}}(t)  - c(t)|
- 2 |b(t)| - 2 \frac{|c_{{1/n_k}}^\prime(t)|}{c_0} 
\right) F_{\lambda, k}(t),
\end{align*}
for almost every $t \in (0,\infty)$, where 
$c_0 = \inf_{t \ge 0}c(t) (> 0)$.  
Hence, observing \eqref{c-cdelta}, \eqref{cdelta_der} and the absolute continuity of $F_{\lambda, k}(t)$ with respect to $t$, 
we obtain
\begin{align*}
F_{\lambda, k}(t) 
\ge F_{\lambda, k}(S) \exp 
\Bigl(
-  &\frac{2 \sqrt{\lambda}}{c_0 n_k} \sup_{t \ge 0} c(t)  \Var(c;[0,\infty))
\\
&- 2 \Norm{b}{{L^1(0,\infty)}} - \frac{2}{c_0} \Var(c;[0,\infty)) 
\Bigr) 
\end{align*}
for every $t \ge S$. 
Letting $k \to \infty$ in the inequality above,  and observing \eqref{cn}, we obtain
$$
F_{\lambda}(t) \ge F_{\lambda}(S) \exp \left(
- 2 \Norm{b}{{L^1(0,\infty)}} - \frac{2}{c_0} \Var(c;[0,\infty)) \right),  
$$
for every $ t \ge S$ satisfying  $t \notin N_1$.  
Letting $\lambda \to \infty$ in the above inequality yields
$$
F(t) \ge F(S) \exp \left(
- 2 \Norm{b}{{L^1(0,\infty)}} - \frac{2}{c_0} \Var(c;[0,\infty)) \right),
$$
for every $ t \ge S$ satisfying  $t \notin N_1$, which together with \eqref{ES} implies that 
\begin{equation}\label{yinfno0}
( y_{1,\infty},  y_{2,\infty}) \neq (0,0). 
\end{equation}

We next prove
\begin{equation}\label{climit2}
c_* = c_\infty. 
\end{equation}
By the expression \eqref{vexp}, we have
\begin{align}
\label{vVnorm}
\norms{\A^{1/2} v(t)}
&= 
\norms{\phi} + \norms{\psi} 
+ 2 \operatorname{Re} (e^{2i c_* t A^{1/2}}\phi, \psi )_H
\\
\label{vprimeVnorm}
\norms{v^\prime(t)} 
&= c_* ^2   \Big(\norms{\phi} + \norms{\psi}
- 2 \operatorname{Re} (e^{2i c_* t A^{1/2}}\phi, \psi )_H
\Big).
\end{align}
By Lemma \ref{lem2} with $g(t) \equiv 2 c_*$, we have  
\begin{equation}\label{heikinproof}
\begin{aligned}
\left|\lim_{T\to\infty}\frac{1}{T}\int_0^T \bigl(e^{2i c_* t A^{1/2}}\phi, \psi \bigr) _H \,dt
\right|
&= \left|\left(\lim_{T\to\infty}\frac{1}{T}\int_0^T e^{2i c_* t A^{1/2}}\phi \,dt, \psi \right)_H \right|
\\
&\le \left\|\lim_{T\to\infty}\frac{1}{T}\int_0^T e^{2i c_* t A^{1/2}}\phi \,dt\right\|
\|\psi\|
= 0.
\end{aligned}
\end{equation}
Thus,
\begin{equation*}
\lim_{t \to \infty} \frac{1}{T} \int_0^T
(e^{2i c_* t A^{1/2}}\phi, \psi) _H dt = 0,
\end{equation*}
which together with \eqref{vVnorm} and \eqref{vprimeVnorm} yields
\begin{equation}\label{cinftyalphabeta}
c_*^2 
\lim_{t \to \infty} \frac{1}{T} \int_0^T \norms{\A^{1/2} v(t)} dt
= \lim_{t \to \infty} \frac{1}{T} \int_0^T \norms{v^\prime(t)} dt. 
\end{equation}

Put
\begin{equation*}
\begin{aligned}
\begin{pmatrix}
w_1(t) \\
w_2(t)
\end{pmatrix}
& :=Y(t) \mathbf{y}_\infty
= \begin{pmatrix}
e^{i \tau(t) A^{1/2}} y_{1,\infty} + e^{-i \tau(t)  A^{1/2} } y_{2,\infty}  \\
i c(t)  e^{i \tau(t) A^{1/2} } y_{1,\infty}   - i c(t) e^{-i \tau(t) A^{1/2}} y_{2,\infty}
\end{pmatrix}.
\end{aligned}
\end{equation*}
Then
\begin{align}
\label{wnorm}
\norms{w_1(t)}
&= 
\norms{y_{1,\infty}} + \norms{y_{2,\infty}} 
+ 2 \operatorname{Re} (e^{2i \tau(t) A^{1/2}}y_{1,\infty}, 
y_{2,\infty})_H,
\\
\label{w2norm}
\norms{w_2(t)} 
&= c(t)^2 
\Big(\norms{y_{1,\infty}} + \norms{y_{2,\infty}} 
- 2 \operatorname{Re} (e^{2i \tau(t) A^{1/2}}y_{1,\infty}, y_{2,\infty})_H \Big).
\end{align}
Using Lemma \ref{lem2} with $g(t) = 2c(t)$, we have in the same way as in 
\eqref{heikinproof},  
\begin{equation*}
\lim_{t \to \infty} \frac{1}{T} \int_0^T
(e^{2i \tau(t) A^{1/2}}y_{1,\infty}, y_{2,\infty})_H dt = 0.
\end{equation*}
Thus \eqref{wnorm}, \eqref{w2norm},  \eqref{yinfno0} and the 
convergence $c_\infty = \lim_{t \to \infty} c(t)$ 
yield
\begin{equation}\label{w1w2}
\begin{aligned}
 c_\infty^2 \lim_{T \to \infty} \frac{1}{T}\int_0^T \norms{w_1(t)}dt  
&= \lim_{T \to \infty} \frac{1}{T}\int_0^T \norms{w_2(t)}dt
\\
& = c_\infty^2
\Big(\norms{y_{1,\infty}} + \norms{y_{2,\infty}}  \Big) \neq 0.
\end{aligned} \end{equation} 
From the expression \eqref{uexp1} with \eqref{rest} and the boundedness of the operator $Y(t)$ uniformly to $t \ge 0$,  it follows that  
$$
\lim_{t \to \infty} \norm{\A^{1/2}u(t) - w_1(t)} + \lim_{t \to \infty}  \norm{u^\prime(t) - w_2(t)} = 0, 
$$
which together with \eqref{w1w2} yields
 \begin{equation*}\begin{aligned}
 c_\infty^2 \lim_{T \to \infty}\frac{1}{T}\int_0^T \norms{\A^{1/2} u(t)}dt  
 &= \lim_{T \to \infty} \frac{1}{T}\int_0^T \norms{u^\prime(t)}dt 
 \neq 0.
  \end{aligned}
 \end{equation*} 
 The equality above and  \eqref{limit} imply
 \begin{equation}\begin{aligned}\label{uc}
 c_\infty^2 \lim_{T \to \infty}\frac{1}{T}\int_0^T \norms{\A^{1/2} v(t)}dt  
 &= \lim_{T \to \infty} \frac{1}{T}\int_0^T \norms{v^\prime(t)}dt 
\neq 0. 
  \end{aligned}
 \end{equation} 
Comparing \eqref{cinftyalphabeta} and \eqref{uc}, we obtain  \eqref{climit2}. 

Now we prove \eqref{c-cinf} under the assumption 
\begin{equation}\label{y1infno0}
y_{1,\infty} \neq 0.
\end{equation} 
The case $y_{1,\infty} =0$ and $y_{2,\infty} \neq 0$ can be treated in the same way. 
Put
$$
f(t) := \int_0^t (c(s) - c_\infty)ds = \tau(t) - c_\infty t \quad \text{for}\quad  t \ge 0.
$$
Then $f \in C([0,\infty))$. We put
$$
\alpha = \liminf_{t \to \infty} f(t), \; 
\beta = \limsup_{t \to \infty}f(t) \; (\in [-\infty,\infty] ).  
$$
It suffices to show  
\begin{equation}\label{conclusion}
\alpha = \beta \in R.
\end{equation}

First we show that $\beta < \infty$. 
Suppose that $\beta = \infty$. 
Since $f$ is continuous and Lebesgue measure of $N_1$ is zero, we can take sequences $\{ t_k \}_{k \in \N}$ such that 
\begin{equation*}
\begin{aligned}
&t_k \notin N_1, \quad \lim_{k \to \infty} t_k  = \infty, 
\quad  \lim_{k \to \infty} f(t_k) = \infty.
\end{aligned}\end{equation*}     
Let $\gamma$ be an arbitrary positive number. 
For every $k \in \N$, since 
$\lim_{n \to \infty}f(t_{k+n}) = \infty$, 
the intermediate value theorem implies that there is $s_k > t_k$ satisfying
$$
f(s_k) = f(t_k) + \gamma.
$$
By using the continuity of $f$ at $s_k$ and the fact that measure of $N_1$ is zero, we can take $r_k$ such that 
\begin{equation}
\label{str}
r_k \notin N_1 \quad r_k > t_k, \quad | f(r_k) - f(t_k) - \gamma| = | f(r_k) - f(s_k)| < \frac{1}{k}.
\end{equation}
By \eqref{climit2}, Lemma \ref{equiv} yields \eqref{y1limit}. This implies
$$
\lim_{t \notin N_1, t \to \infty} e^{{-i f(t) A^{1/2}}} \phi = 
\lim_{t \notin N_1, t \to \infty} e^{{-i (\tau(t) - c_\infty t) A^{1/2}}} \phi  = y_{1,\infty} \quad \text{in} \quad H,
$$
since $e^{i t A^{1/2}}$ is a unitary operator on $H$. 
Hence, letting $k \to \infty$ in the equality 
\begin{equation*}
e^{i  (f(r_k) -f(t_k) - \gamma) A^{1/2}}
e^{-i  f(r_k) A^{1/2}} 
\phi 
= e^{-i \gamma A^{1/2}} e^{- i f(t_k) A^{1/2}} \phi, 
\end{equation*}
and observing \eqref{str} and the continuity of the unitary operator $e^{is A^{1/2}}$ with respect to $s$, we obtain
$$
y_{1,\infty} = e^{-i \gamma A^{1/2}} y_{1,\infty}. 
$$
Thus, we have
\begin{equation*}
(I + A^{1/2})^{-1} y_{1,\infty} = e^{-i \gamma A^{1/2}} (I + A^{1/2})^{-1} y_{1,\infty}.
\end{equation*} 
Since $\gamma > 0$ is arbitrary, and since 
$(I + A^{1/2})^{-1}y_{1,\infty} \in D(A^{1/2})$, we differentiate the equality above with respect to $\gamma$ to obtain 
$$
0 = 
\od{}{\gamma} e^{-i \gamma A^{1/2}} (I + A^{1/2})^{-1} y_{1,\infty}
= - i  A^{1/2} e^{-i \gamma A^{1/2}} (I + A^{1/2})^{-1} y_{1,\infty}
$$
on $(0, \infty)$. This implies that $y_{1,\infty} = 0$ by the injectivity of $A^{1/2}$ and $e^{-i \gamma A^{1/2}}$,   
which contradicts \eqref{y1infno0}.  
                                                       
The assumption $\alpha =-\infty$ deduces contradiction in the same way. 

We finally prove \eqref{conclusion}. 
The above facts imply that $\alpha, \beta \in \R$. 
Suppose that \eqref{conclusion} fails to hold. Then 
the interval $(\alpha, \beta)$ is not emplty.   
Let $\gamma $ be an arbitrary number $\gamma \in (\alpha, \beta)$. 
For every $k \in \N$,  the intermediate value theorem implies that there exists $s_k > k$ satisfying 
$f(s_k) = \gamma$. 
Then by the same reason as \eqref{str}, we can take $r_k$ such that 
\begin{equation}
\label{rk}
r_k \notin N_1 \quad r_k > k, \quad | f(r_k) - \gamma| = | f(r_k) - f(s_k)| < \frac{1}{k}.
\end{equation}
Letting $k \to \infty$ in the equality
\begin{equation*}
e^{-i  (f(r_k)  - \gamma) A^{1/2}}
e^{i  (\tau(r_k) - c_\infty r_k) A^{1/2}} 
y_{1,\infty} 
= e^{i \gamma A^{1/2}} y_{1,\infty}, 
\end{equation*}
and observing \eqref{y1limit}, \eqref{rk} and the continuity of $e^{i t A^{1/2}}$ with respect to $t$,  we obtain\begin{equation*}
\phi = e^{i \gamma A^{1/2}} y_{1,\infty}.
\end{equation*} 
Hence we have 
\begin{equation}\label{gamma1}
 (I + A^{1/2})^{-1} \phi = e^{i \gamma A^{1/2}} (I + A^{1/2})^{-1}y_{1,\infty}. 
\end{equation}  
Since $\gamma \in  (\alpha, \beta)$ is arbitrary and since $(I + A^{1/2})^{-1}H \subset D(A^{1/2})$, 
we differentiate \eqref{gamma1} with respect to $\gamma $ to obtain 
$$
i A^{1/2} e^{i \gamma A^{1/2}} (I + A^{1/2})^{-1} y_{1,\infty} 
= \od{}{\gamma} e^{i \gamma A^{1/2}} (I + A^{1/2})^{-1} y_{1,\infty} = 0
$$
on $(\alpha, \beta)$. This implies that $y_{1,\infty} = 0$ by the injectivity of $A^{1/2}$ and 
$e^{i \gamma A^{1/2}}$,  which contradicts to \eqref{y1infno0}.  
\hfill \qedsymbol

\section{ Appendix}

In the case $b(t)$ is an integrable $C^1$ function and $c$ is a $C^1$ function satisfying \eqref{cass1}, it is clear that there exists a unique solution of initial value problem \eqref{WES}--\eqref{initialS}, equivalently, 
\eqref{we0}--\eqref{initial0}. Namely, the following proposition holds. 

\begin{propA} 
Let $b(t)$ be an integrable $C^1$ function on $[0,\infty) $. 
Let $c(t)$ be a $C^1$ function satisfying \eqref{cass1}. 
Then for every $(\phi_0, \psi_0) \in \HH_{1/2} \times H$,  
the Cauchy problem \eqref{we0}--\eqref{initial0} has a unique global weak solution.  
Furthermore, if $(\A^{1/2}\phi_0, \psi_0) \in D(A^{J/2}) \times D(A^{J/2})$ for $J \ge 1$, the following holds. 
\begin{align*}
&(\A^{1/2} u,   u^\prime) \in \bigcap_{j = 0,1} C^j \left( [0,\infty); \HHH_{(J-j)/2} \times \HHH_{(J-j)/2} \right).
\end{align*}
\end{propA}

On the existence of solutions of the Cauchy problem \eqref{WE}--\eqref{initial} under the assumption that $c(t)$ is of bounded variation, there are some results. 
Colombini, De Giorgi and Spagnolo \cite{CGS} showed the existence of solution 
\begin{alignat*}{2} 
&\pd{^2 u}{t^2}(t,x)  - \sum_{i,j= 1}^n 
a_{i,j}(t)\pd{^2 u}{x_i \partial x_j} 
 = f(t,x) 
&&
\text{ in  }[0,\infty) \times \R^n,
\end{alignat*}
in the class 
$u \in C([0,T], H^{s+1}_{\loc})$,  $\partial u/\partial t \in  L^2([0,T],H^s_{\loc}) $ and 
$\partial^2 u/\partial t^2 \in L^1([0,T],H^{s-1}_{\loc})$, 
where $a_{i,j}(t)$ is of bounded variation and 
$$
a_{i,j}(t) = a_{j,i}(t), \quad 
\sum_{i,j= 1}^n a_{i,j}(t) \xi_i \xi_j \ge \lambda_0 |\xi|^2 
\; \text{for all } \xi \in \R^n,
$$  
for $\lambda_0 > 0$. In the case $A$ is a corecive self-adjoint operator, De Simon and Torelli 
\cite{STo} showed the unique existence of the solution 
of \eqref{WE}--\eqref{initial} in the class 
$u \in W^{1,2}([0,T], H),  \; \partial u/\partial t \in L^2([0,T], D(A^{1/2})) $. 
Arosio \cite{A} considered  
\eqref{wave1}--\eqref{wave3}  with $(\phi_0, \psi_0) \in H^{1}_0(\Omega) \times L^2(\Omega)$ for bounded domain $\Omega$,  and 
 showed the unique existence of solution 
 in the class $u \in C([0,\infty), H^{1}_0(\Omega)) \cap C^1([0,\infty),L^2(\Omega))$. 
 The results above (\cite{CGS}, \cite{STo} and \cite{A}) considered the solutions in the sense of  distribution with respect to $t$. 
 On the other hand, B\'{a}rta \cite[section 2]{Ba} considered the hyperbolic equation 
\begin{alignat}{2}	\label{baeq}
&\pd{^2 u}{t^2} = \sum_{i,j = 1}^n \pd{}{x_i} \left(a_{i,j}(t,x)\pd{u}{x_j} \right)(t,x)
&&+ \sum_{i = 1}^n p_i(t,x)  \pd{u}{x_i}(t,x) + q(t,x) u(t,x)
\nonumber
\\
&  
\qquad \quad \qquad  \qquad 
&&\text{ in } (0,T) \times \Omega,
\\
	\label{baerinitial}
&	u(0,x) = \phi_0(x), \quad \pd{u}{t}(0,x) = \psi_0(x) \; 
&&	\text{ in } \Omega, 
\end{alignat}
where $\Omega$ is a bounded domain in $\R^n$, 
and $a_{i,j}$, $p_i$ and $q$ are functions satisfying the following:
\begin{align*}
&a_{i,j} \in BV([0,\infty), W^{1,\infty}) \cap L^\infty([0,\infty),Lip(\Omega)), 
\\
&a_{i,j}(t,x) = a_{j.i}(t,x), \quad 
\sum_{i,j= 1}^n a_{i,j}(t,x) \xi_i \xi_j \ge \lambda_0 |\xi|^2 
\; \text{for all } \xi \in \R^n,
\\
&p_i, q \in BV([0,\infty), L^{\infty}).   
\end{align*}
Then he showed the unique existence of the solution 
	$u(t) \in C([0,\infty);H^1_0(\Omega)) \cap C^1([0,\infty);L^2(\Omega))$ of  \eqref{baeq} with initial value in 
	$(H^2(\Omega) \cap H^1_0(\Omega) ) \times H^1_0(\Omega)$, such that for an at most countable subset $N$,  
		$$
		(u(t), u^\prime(t)) \in C([0,\infty) \setminus N; (H^2(\Omega) \cap H^1_0(\Omega)) \times H^1_0(\Omega) ),
		$$	
and $u^\prime(t)$ is differentiable with values in $L^2(\Omega)$ at $t \in [0,\infty) \setminus N$.  
	B\'{a}rta \cite{Ba} proved this by showing and applying an abstract theorem. 
	
	\begin{themB}[Proposition 1.3 and Corollary 1.4 of \cite{Ba}]
	Let $X$, $Y$ be uniformly convex Banach space. 
	Let $\{ \mathbb A(t) \}_{t \ge 0}$ be the family of closed operators in $X$ with domain $D(\mathbb A(t)) \equiv Y$.  
	Assume that the following conditions {\rm(\rom 1)--(\rom3)} hold. 
	\begin{enumerate}
		\item For every $t \ge 0$,  $D(\mathbb A(t))$ is dense in $X$, and $\{ \mathbb A(t) \}_{t \ge 0}$  is stable with constants $\beta, 1$, that is, the semi-infinite interval $(\beta, \infty)$ belongs to the resolvent set of $- \mathbb A(t)$ and 
	$$
		\Norm{(\mathbb A(t) + \xi)^{-1}}{{\LL(X)}} \le (\xi - \beta)^{-1}, \quad \xi > \beta, 
		$$		
		for every $t \ge 0$. 		
		\item There exists a family of uniformly convex Banach spaces $X_t = (X, \Norm{\cdot}{t})$ and a function of bounded variation $a:[0,\infty)\to \R $ such that
		\begin{equation*}
		\frac{\Norm{x}{t}}{\Norm{x}{s}} \le e^{|a(t) - a(s)|}
		\end{equation*}
		holds for all $x \in X$ and $0 \le s,t \le T$.
		\item The mapping $t \mapsto \mathbb A(t)$ is of bounded variation with values in $B(Y,X)$. 
			\end{enumerate}
Then there exists a family operators $U(t,s) \in B(X)$, $(t,s) \in \Delta=\{(t,s) \in \R^2; 0 \le s \le t \le T \}$ such that the following {\rm(a)--(c)} hold. 

\renewcommand{\labelenumi}{(\alph{enumi})}
\begin{enumerate}
				\item $U(t,s)$ is strongly continuous in $X$ with respect to $s,t$, $U(t,t) = I$ and $\Norm{U(t,s)}{X} \le e^{\beta(t-s)}$.
		\item $U(t,s)Y \subset Y$ and $\Norm{U(t,s)}{Y} \le e^{\beta(t-s)}$.
		\item For every $y \in Y$, 
		there exists a countable set $N_{y} \subset [0,\infty)$ such that 
		the mapping
		$t \mapsto U(t,s)y$ is continuous 
		in the norm of $Y$, and that $D_t U(t,s)y = - \mathbb A(t)U(t,s)y$ holds for all  $(t,s) \in \Delta$,  
$ t \notin N_{y}$.
	\end{enumerate}
	\end{themB}
	
As is stated above, B\'{a}rta \cite{Ba} applied Theorem B to the hyperbolic equation \eqref{baeq} to show the unique existence of solutions.  
Similarly, we can apply Theorem B to the Cauchy problem \eqref{we0}--\eqref{initial0}  to obtain the solution $u(t)$.  
 In the argument of this paper, we need the fact that $u^\prime(t)$ is absolutely continuous with value in $\HHH_{-1/2}$. 
 This fact is verified by the following lemma, which is proved at the end.  

\begin{lemma}\label{absolute}
Let  $X$  be a separable Banach space.  
Assume that  $f(t)$  is an  $X$-valued continuous function on  $[a,b]$  and that  $g(t)$  is an  $X$-valued 
integrable function on  $(a,b)$.  
Assume moreover that there exists an at most countable subset  $N$  of  $[a,b]$  such that  
$g(t)$  is continuous on  $[a,b]\setminus N$  and that  $f(t)$  is differentiable 
on $ (a,b)\setminus N$ with $f^\prime(t) = g(t)$. 
 Then  $f(t)$  is absolutely continuous on  $[a,b]$,  and satisfies
$$
f(t) = \int_a^t g(s) ds + f(a) \quad \text{ for }\;t \in [a,b]. 
$$  
\end{lemma}

Now we state a proposition on the unique existence of the solution \eqref{we0}--\eqref{initial0}. 
	\begin{prop} 
		\label{prop:global1}
			Let $b(t)$ be of bounded variation and integrable on $[0,\infty) $. 
		Let $c(t)$ be of bounded variation on $[0,\infty) $ satisfying \eqref{cass1}. 
		Then the following assertions hold. 
		For every $(\phi_0, \psi_0) \in \HH_{1/2} \times D(A^{1/2})$ satisfying  
			$\A^{1/2} \phi_0 \in D(A^{1/2})$, the Cauchy problem \eqref{we0}--\eqref{initial0} has a unique global weak  solution.  
			Furthermore,  
			$u^\prime \in AC_{\loc}([0,\infty);H)$ and 
			there exists an at most countable subset $N$ such that   
			$$
			(\A^{1/2}u(t), u^\prime(t)) \in C([0,\infty) \setminus N;\HHH_{1/2} \times \HHH_{1/2} ) ,
			$$	
and 
\eqref{we0} holds in the space $H$ at every $t \in [0,\infty) \setminus N$. 
	\end{prop}
	
	\begin{proof}
	Let
	$$
	X_t  \equiv X =  H \times H, \quad Y = \HHH_{1/2} \times \HHH_{1/2}, 
	$$
	with inner product on $X_t$
	$$
	\begin{pmatrix}
	\begin{pmatrix}
	x_1 \\
	x_2 
	\end{pmatrix}, 
	\begin{pmatrix}
	y_1 \\
	y_2 
	\end{pmatrix} 
	\end{pmatrix}_t
	:= c(t)^2 (x_1, y_1)_H + (x_2,y_2)_H. 
		$$
We define 
	$$
	\mathbb A(t) = \begin{pmatrix} 
	0     & - A^{1/2} \\
	c(t)^2 A^{1/2} & b(t)   
	\end{pmatrix} \; 
	\text{ with domain } \;
	D(\mathbb A(t)) = Y. 
	$$
		Then in the same way as in the proof of \cite[section2]{Ba}, 
	we see that the assumption of Theorem B are satisfied. 
	Let $\{U(t,s) \in B(H \times H); 0 \le s \le t \}$ be a family of evolution operators given by Theorem B. 
	Put 
	$$
	\mathbf{x}(t) = \begin{pmatrix}
	w(t) \\
	v(t) 
	\end{pmatrix}
	:= U(t,0)
	\begin{pmatrix}
	\A^{1/2} \phi_0 \\
	\psi_0 
	\end{pmatrix}.
	$$
	Then
\begin{equation*}
\mathbf{x}(0) = \begin{pmatrix}
		\A^{1/2}\phi_0 \\
		\psi_0 
	\end{pmatrix} \in Y, 
\end{equation*}
and thus, Theorem B implies $\mathbf{x}(t)  \in Y $ with 
\begin{align}\label{xnorm}
& \Norm{\mathbf{x}(t)}{Y} \le \Norm{U(t,0)}{{\LL(Y)}} 
\Norm{
	\begin{pmatrix}
	\A^{1/2} \phi_0 \\
\psi_0 
\end{pmatrix}
}{Y}
\le  e^{\beta t}
\Norm{
	\begin{pmatrix}
	\A^{1/2} \phi_0 \\
	\psi_0 
	\end{pmatrix}
}{Y}
\end{align}
for every $t \ge 0$, and there exists at most countable set $N_0$ depending on initial data such that 
\begin{align}\label{xbelong}
&	\mathbf{x}(t)
	\in C\left([0,\infty);H \times H\right) \cap 
	C\left([0,\infty) \setminus N_0\,;\, \HHH_{1/2} \times \HHH_{1/2} \right), 
\end{align}
	and that $\mathbf{x}(t) $ is differentiable  on $[0,\infty) \setminus N_0$  and satisfies 
	\begin{align*}
	&\dfrac{d}{dt} \mathbf{x}(t)
+ \mathbb A(t) \mathbf{x}(t)
	=  
	\begin{pmatrix}
	0 \\
	0 
	\end{pmatrix} \; 
	\text{in}\; H \times H, 
	\; t \in [0,\infty) \setminus N_0.
	\end{align*}
	Since $c$ is of bounded variation on $[0,\infty)$, there is an at most countable set $N_c \subset [0,\infty)$ such that $c \in C([0,\infty) \setminus N_c)$. Thus, by \eqref{xnorm} and 	
	\eqref{xbelong}, we see that 
\begin{align*}
&	\mathbb A(t) \mathbf x(t) \in 
C\left([0,\infty) \setminus ( N_0 \cup N_c); H \times H \right) 
\cap L^1_{\loc}([0,\infty); H \times H).
		\end{align*}
Hence, 
\begin{equation}
\label{ode1}
\od{}{t}\mathbf x(t) = - \mathbb{A}(t) \mathbf{x}(t) 
\in C \left( [0,\infty) \setminus ( N_0 \cup N_c); H \times H \right) \cap L^1_{\loc}([0,\infty); H \times H).
\end{equation}
This fact and \eqref{xbelong}  with the aid of Lemma \ref{absolute} imply $\mathbf{x}(t) \in AC_{\loc}([0,\infty); H \times H)$. 
We define
\begin{equation}
\label{udef}
u(t) := \int_0^t v(s) ds + \phi_0.
\end{equation}	
Then
\begin{equation*}
u^\prime(t) = v(t) \in AC_{\loc}([0,\infty);H). 
\end{equation*}	
Since $v(t)$ is bounded in $\HHH_{1/2}$  by \eqref{xnorm},
we see that $u(t) - \phi_0 \in AC_{\loc}([0,\infty); \HHH_{1/2})$. 
Since $w$ is absolutely continuous, 
\begin{equation}\label{ode3}
\begin{aligned}
\A^{1/2} u(t) &= \int_0^t \A^{1/2} v(s) ds + \A^{1/2}\phi_0
= \int_0^t w^\prime(s) ds + \A^{1/2}\phi_0
= w(t),
\end{aligned}
\end{equation}
for every $t \in [0,\infty)$.  
From \eqref{ode1}--\eqref{ode3}, it follows that  		
$u$ satisfies \eqref{we0} in $H$ for all $t \in (0,\infty) \setminus (N_0 \cup N_c)$. 
	From the argument above, we see that $u$ is a weak solution of  \eqref{we0}--\eqref{initial0} and belongs to the class stated in Proposition \ref{prop:global1}. 

The uniqueness of the solution is easily seen by Gronwall's inequality. 

\end{proof}	

\begin{proof}[Proof of Lemma \ref{absolute}]
	Fix an arbitrary positive number  $\varepsilon$.  Since  $g(t)$  is integrable,  
	there exists a positive number  $\gamma$  such that the estimate
	\begin{equation}\label{gamma}
	\int_E \|g(t)\|\,dx < \frac{\varepsilon}{4} \text{ if } \mu(E) < \gamma
	\end{equation}
holds,	where $\mu(E)$ denotes the Lebesque measure of $E$ for Lebesque measurable set $E \subset \R$. 
	Put  $a = t_1$,  $b = t_2$  and  $N = \{t_j \mid j = 3,4,\dots\}$.  Since  
	$f(t)$  is uniformly continuous on  $[a,b]$,  there exists a positive number  
	$\delta_j < \gamma/{2^{j+1}}$  for every  $j = 1,2,\dots$  such that the 
	estimate  $\|f(t) - f(s)\| < \varepsilon/2^{j+2}$  holds for every  
	$s$,~$t \in [a,b]$  satisfying  $|t-s| < 2\delta_j$.  
	On the other hand, for every  
	\[
	c \in S := (a,b) \setminus \bigcup_{j=1}^\infty (t_j-\delta_j,t_j+\delta_j), 
	\]
	the function  $f(t)$  is differentiable at  $t = c$,  and  
	$f^\prime(t) = g(t)$  is continuous at  $t = c$.  Hence there exists a positive number  $\delta(c)$  
	such that the inequalities
	\begin{equation}\label{fg}
	\big\|f(t) - f(c) - (t-c)g(c)\big\| 
	\le \frac{\varepsilon|t-c|}{8(b-a)}, \quad
	\|g(t)-g(c)\| \le \frac{\varepsilon}{8(b-a)}
	\end{equation}
	hold for every  $t \in \bigl(c-\delta(c),c+\delta(c)\bigr) \cap [a,b]$. 
	Then we have  
	\[
	\bigcup_{j=1}^\infty (t_j-\delta_j,t_j+\delta_j) \cup
	\bigcup_{c\in S} \bigl(c-\delta(c),c+\delta(c)\bigr) 
	\supset [a,b].
	\]
	Hence we can choose a finite subset  $J_0$  of  $\mathbb{N}$  and a finite 
	sequence  $\{c_k \in S \}_{k=1}^M$  satisfying  
	$a < c_1 < c_2 < \dots < c_M < b$ such that  
	\begin{equation}\label{minimal}
	\bigcup_{j\in J_0} (t_j-\delta_j,t+\delta_j) \cup 
	\bigcup_{k=1}^M \bigl(c_k-\delta(c_k),c_k+\delta(c_k)\bigr) 
	\supset [a,b]. 
	\end{equation}
	Let  $(J, K)$ be a minimal pair of set such that $J  \subset J_0, \; K \subset \{1,2,\dots,M\}$  satisfying  
	\[
	\bigcup_{j\in J} (t_j-\delta_j,t_j+\delta_j) \cup
	\bigcup_{k\in K} \bigl(c_k-\delta(c_k),c_k+\delta(c_k)\bigr)
	\supset [a,b].
	\]
	Put
	\begin{align*}
	\mathcal{I} 
	&= \{(\alpha_m,\beta_m) \mid m = 1,\dots,L\} \\
	&= \bigl\{(t_j-\delta_j,t_j+\delta_j)\mid j \in J\bigr\}\cup
	\bigl\{\bigl(c_k-\delta(c_k),c_k+\delta(c_k)\bigr)\mid k\in K\bigr\}.
	\end{align*}
Renumbering if necessary, we can assume that 
$$
\alpha_m <  \alpha_{m + 1},\quad \beta_m < \beta_{m+1}
$$ 	
for $m = 1, 2,\dots,L-1$. 
By the minimality, we see that 
$\alpha_1 < a  \le \alpha_2$, 
$\beta_{L-1} \le b <  \beta_L$. 
We also have
	\begin{equation*}
	\alpha_m < \beta_{m-1} \le \alpha_{m+1} < \beta_m  
	\end{equation*}		
for every $ m = 2,\dots,L-1$. 
In fact, if $ \alpha_{m+1} < \beta_{m-1}$, then we have 
$$
\alpha_{m-1} < \alpha_m < \alpha_{m +1} < \beta_{m-1} < \beta_m < \beta_{m+1}.
$$
It follows that 
$(\alpha_m, \beta_m) \subset 
(\alpha_{m-1}, \beta_{m-1}) 
\cup (\alpha_{m+1}, \beta_{m+1})$, 
which contradicts the minimality of  $	\mathcal{I}$. 

	We now choose  a sequence  $\{p_m\}_{m=0}^L$  satisfying  
	$a = p_0 < p_1 < \dots < p_L = b$  such that   
	$\alpha_{m+1} < p_m < \beta_{m}$  holds for every  
	$m = 1,\dots,L-1$.  
			Here we note 
		\begin{equation}
		\label{alphapbeta}
		\alpha_m < p_{m-1} 
		< p_m < \beta_m  
		\end{equation}
		for every $m = 1,\dots,L$. 	
			Furthermore, we can choose  $\{p_m\}_{m=1}^{L-1}$  
	so that  $p_{m-1} \le c_k \le p_m$  holds if  
	$(\alpha_{m},\beta_{m})$  is of the form $\bigl(c_k - \delta(c_k),c_k + \delta(c_k)\bigr)$.  
Put	
	\begin{align*} 	
	\Lambda &:= \bigl\{m \mid  (\alpha_{m},\beta_{m})
	= (t_{j(m)}-\delta_{j(m)},t_{j(m)}+\delta_{j(m)})
	\text{ with some } j(m)\bigr\},
	\\
	P &:= \bigl\{m \mid 
	(\alpha_{m},\beta_{m})
	= \bigl(c_{k(m)}-\delta(c_{k(m)}),c_{k(m)}+\delta(c_{k(m)})\bigr)
	\\* 
	&\qquad \qquad \qquad \qquad \qquad \qquad     \text{ with some } k(m) \in K \bigr\}.
	\end{align*}	
	Then we have
	\begin{multline} \label{sum0}
	\left\|f(b)-f(a)-\int_a^b g(s)\,ds\right\| 
	\\
	= \left\|\sum_{m=1}^L\bigl\{f(p_m)-f(p_{m-1})\bigr\} 
	- \int_a^b g(s)\,ds\right\|
	\le I_1 + I_2 + I_3\,,
	\end{multline}
	where
	\begin{align*}
	I_1 &= \sum_{m \in \Lambda}\|f(p_m)-f(p_{m-1})\|, 
	\\
	I_2 &= \left\|\int_{E} g(s)\,ds\right\|
	\quad \text{ with }
	E = \bigcup_{m\in\Lambda}[p_{m-1},p_m],
	\\
	I_3 &= \sum_{m \in P} 
	\left\|f(p_m)-f(p_{m-1})-\int_{p_{m-1}}^{p_m}g(s)\,ds\right\|.
	\end{align*}
Observing \eqref{alphapbeta}, we have
	\begin{equation} \label{sum1}
	I_1
	< \sum_{j=1}^\infty \sup_{t,s \in (t_j-\delta_j,t_j+\delta_j)} \|f(t)-f(s)\|
	< \frac{\varepsilon}{4}\,.
	\end{equation}
	Next, since  
	\[
	\mu(E) 
	= \sum_{m\in\Lambda} (p_m-p_{m-1})
	< \sum_{m\in\Lambda} (\beta_{m}-\alpha_{m})
	< \sum_{j=1}^\infty 2\delta_j 
	< \gamma,
	\]
	inequality \eqref{gamma} implies
	\begin{equation} \label{sum2}
	I_2 = \left\|\int_E g(s)\,ds\right\| < \frac{\varepsilon}{4}\,.
	\end{equation}
	Finally, we treat the case that  $m \in P$, that is, 
	$\bigl(\alpha_{m},\beta_{m}\bigr) = 
	\bigl(c_k-\delta(c_k),c_k+\delta(c_k)\bigr)$  
	holds with some  $k = k(m) \in K$.  In this case we have  
	\[
	c_k-\delta(c_k) < p_{m-1} \le c_k \le p_m < c_k+\delta(c_k).
	\]
	Then observing \eqref{fg}, we have 
	\begin{align*}
	&    \left\|f(p_m) - f(c_k) - \int_{c_k}^{p_m}g(s)\,ds\right\| \\
	&\le \big\|f(p_m) - f(c_k) - (p_m-c_k)g(c_k)\big\|
	+ \int_{c_k}^{p_m}\|g(s) - g(c_k)\|\,ds \\
	&\le \frac{\varepsilon(p_m - c_k)}{4|b-a|}\,.
	\end{align*}
	In the same way we have
	\[
	\left\|f(c_k) - f(p_{m-1}) - \int_{p_{m-1}}^{c_k}g(s)\,ds\right\|
	\le \frac{\varepsilon(c_k - p_{m-1})}{4|b-a|}\,.
	\]
	Summing up we obtain  
	\[
	\left\|f(p_m) - f(p_{m-1}) - \int_{p_{m-1}}^{p_m}g(s)\,ds\right\|
	\le \frac{\varepsilon(p_m - p_{m-1})}{4|b-a|}
	\]
	for every  $m \in P$,  which implies
	\begin{equation} \label{sum3}
	I_3 \le \frac{\varepsilon}{4(b-a)}\sum_{m\in P}(p_m-p_{m-1})
	\le \frac{\varepsilon}{4}\,.
	\end{equation}
	Substituting  \eqref{sum1},  \eqref{sum2}  and \eqref{sum3}    into \eqref{sum0},  we conclude
	\[
	\left\|f(b)-f(a) - \int_a^b g(s)\,ds\right\| < \varepsilon.
	\]
	Since  $\varepsilon > 0$  is arbitrary,  we have  
	\[
	f(b)-f(a) = \int_a^b g(s)\,ds.
	\]
	Applying the same argument on  $[a,t]$  for every  $t \in [a,b]$,  we 	obtain the conclusion.
\end{proof}

\end{document}